\documentclass{amsart}
\usepackage[margin=1in]{geometry}
\usepackage{amssymb,latexsym,thmtools,amsthm}
\usepackage[dvipsnames]{xcolor}
\usepackage[parfill]{parskip}
\usepackage[all,cmtip]{xy}
\usepackage{url}

\theoremstyle{plain}
\newtheorem{theorem}{Theorem}[section]

\newtheorem{proposition}[theorem]{Proposition}
\newtheorem{lemma}[theorem]{Lemma}

\theoremstyle{definition}
\newtheorem{definition}[theorem]{Definition}

\declaretheorem[sibling=theorem,name=Remark,qed={$\clubsuit$}]{remark}

\newtheorem{question}[theorem]{Question}

\usepackage[colorlinks]{hyperref}
\hypersetup{
    colorlinks,
    linkcolor={ForestGreen},
    citecolor={MidnightBlue},
    urlcolor={Magenta}
}

\renewcommand{\phi}{\varphi}
\renewcommand{\theta}{\vartheta}
\renewcommand{\epsilon}{\varepsilon}

\DeclareMathOperator{\PP}{\mathbb{P}}

\DeclareMathOperator{\calS}{\mathcal{S}}

\newcommand{\codim}{\mathrm{codim}}
\newcommand{\Gr}{\mathbb{G}\mathrm{r}}
\newcommand{\str}{\mathrm{str}}
\DeclareMathOperator{\Sing}{Sing}
\let\Vec\relax\DeclareMathOperator{\Vec}{Vec}
\DeclareMathOperator{\Set}{Set}

\DeclareMathOperator{\im}{im}

\DeclareMathOperator{\NN}{\mathbb{N}}
\DeclareMathOperator{\CC}{\mathbb{C}}

\date{}

\title{The set of forms with bounded strength is not closed}

\author{Edoardo Ballico}
\address[Edoardo Ballico, Alessandro Oneto]{Universit\`a di Trento, Via Sommarive, 14 - 38123 Povo (Trento), Italy}
\email{edoardo.ballico@unitn.it, alessandro.oneto@unitn.it}

\author{Arthur Bik}
\address[Arthur Bik]{MPI for Mathematics in the Sciences, Leipzig, Germany}
\email{arthur.bik@mis.mpg.de}

\author{Alessandro Oneto}
\address{\vspace*{-25pt}}

\author{Emanuele Ventura}
\address[Emanuele Ventura]{Universit\`{a} di Torino, Dipartimento di Matematica, via Carlo Alberto 10, 10123 Torino, Italy}
\email{emanuele.ventura@unito.it, emanueleventura.sw@gmail.com}

\makeatletter
\@namedef{subjclassname@2020}{%
  \textup{2020} Mathematics Subject Classification}
\makeatother

\subjclass[2020]{15A21,13A02,14R20}
\keywords{strength; additive decompositions; polynomial functors; homogeneous polynomials}

\begin{document}
\begin{abstract}
The strength of a homogeneous polynomial (or form) is the smallest length of an additive decomposition expressing it whose summands are reducible forms. Using polynomial functors, we show that the set of forms with bounded strength is not always Zariski-closed. More specifically, if the ground field is algebraically closed, we prove that the set of quartics with strength $\leq3$ is not Zariski-closed for a large number of variables.
\end{abstract}

\maketitle

\section{Introduction}

In \cite{ananyan2020small}, Ananyan and Hochster defined the notion of \textit{strength} of a polynomial to solve the famous conjecture by Stillman on the existence of a uniform bound, independent on the number of variables, for the projective dimension of a homogeneous ideal of a polynomial ring. 
Interestingly, well before this groundbreaking work in commutative algebra, Schmidt \cite[p. 245]{S85} had introduced the very same measure of complexity (called {\it Schmidt rank}) of a polynomial in the context of arithmetic geometry to study integer points in varieties defined over the rationals. We shall use the terminology
of Ananyan and Hochster. \vspace{-2pt}

Let $\calS =  \bigoplus_{d\geq 0}\calS_d$ be a standard graded polynomial ring in $n$ variables, i.e., $\calS_d$ is the vector space of degree-$d$ homogeneous polynomials, or forms, in $n$ variables with coefficients in a field $\Bbbk$. \vspace{-2pt}

\begin{definition}\label{def:strength}
A \textbf{strength decomposition} of a homogeneous polynomial $f$ is an expression of the form
\begin{equation}\label{eq:strength_decomp}
f = g_1h_1 + \ldots + g_rh_r, \quad \text{ where $g_i,h_i$ are homogeneous with } 1\leq \deg(g_i),\deg(h_i) \leq \deg(f)-1.
\end{equation}
We define the \textbf{strength} of $f$ to be the smallest length of a strength decomposition of $f$. Note that this differs from the definition used in \cite{ananyan2020small} by 1, so that a homogeneous polynomial has strength $\leq r$ if and only if it is the sum of $r$ polynomials of strength $\leq 1$.
\end{definition}\vspace{-2pt}

Since its introduction, the notion of strength has been studied in several works, see e.g. \cite{derksen2017topological, kazhdan2018ranks, bik2019polynomials, erman2019big, erman2020big, ananyan2020strength, ballico2020strength, bik2020strength, ballico2021strength, KP0, KP1, KP2} (in the last three works, strength is called Schmidt rank). Despite this interest, our knowledge of  the strength of forms is still quite limited.\vspace{-2pt}

A special type of strength decomposition is the one where each summand in \eqref{eq:strength_decomp} is required to have a linear factor;  such decompositions are called \textbf{slice decompositions}. The smallest length of a slice decomposition of a form $f$ is called its \textbf{slice rank}, or \textbf{qrank} \cite{derksen2017topological} when $f$ has degree $3$. It is well-known that a form $f$ admits a slice decomposition of length $r$ if and only if the hypersurface $\lbrace f = 0\rbrace$ contains a linear subspace of codimension $r$; see for instance \cite[Proposition 2.2]{derksen2017topological}. Hence, in the projective space $\PP\calS_d$ of degree-$d$ homogeneous polynomials, the set of forms having slice rank $\leq r$ is the image of the projection onto the first factor of an incidence variety inside $\PP\calS_d \times \Gr(n-r,n)$, where $\Gr(k,n)$ is the Grassmannian of $k$-dimensional linear spaces in $\Bbbk^n$; see \cite[Example 12.5]{harris1992algebraic}. Thus it is Zariski-closed.\vspace{-2pt}
	
It is natural to ask whether the same holds for sets of forms of bounded strength.\vspace{-2pt}

\begin{question}[{\cite[Example 2]{bik2019polynomials}}]\label{quest}
In the projective space $\PP\calS_d$ of degree-$d$ homogeneous polynomials in~$n$ variables, is the set of polynomials having strength $\leq r$ always Zariski-closed?
\end{question}

In this note, we give a negative answer to this question. Note that, since for quadrics and cubics the notion of strength coincides with that of slice rank, the smallest degree where Question \ref{quest} might have a negative answer is $4$. This is indeed the case: we prove that, for quartics in sufficiently many variables, the space of forms with strength~$\leq3$ is {\it not} closed. Employing the theory of polynomial functors \cite{draisma2019topological, bik2019polynomials, bik2020phd}, in Section \ref{sec:example} we show the following.

\begin{theorem}
Let $\Bbbk$ be algebraically closed. For polynomials $x,y,u,v$ of degree $1$ and $f,g,p,q$ of degree $2$, the polynomial
\begin{equation}\label{counterexample}
x^2f+y^2g+u^2p+v^2q, 
\end{equation}
is always a limit of strength-$\leq3$ polynomials, but for a sufficiently large number of variables and a suitable choice of $x,y,u,v,f,g,p,q$ it has strength $4$.
\end{theorem}

This shows that the answer to Question~\ref{quest} is negative for $d=4$, $r=3$ and $n\gg0$. We leave the question concerning the minimal $n$ where this is possible open. Note that over $\CC$ it has to be at least $6$: indeed, the strength of a form is bounded above by its slice rank and the slice rank of a degree-$4$ form in $n\leq 5$ variables is at most $3$; see e.g. \cite{bik2020strength}.

\begin{question}\label{quest1}
What is the smallest number of variables $n$ where \eqref{counterexample} can have strength $4$? Does the sum of squares $q = x^2s^2 + y^2t^2 + u^2w^2 + v^2z^2$ in $8$ variables possess this property?
\end{question}

The degree $d = 4$ of our counterexample is minimal. However, we do not know whether the strength $r = 3$ is minimal as well. This leads to the following question.

\begin{question}\label{quest2}
In the projective space $\PP\calS_d$ of degree-$d$ homogeneous polynomials in~$n$ variables, is the set of polynomials having strength $\leq 2$ always Zariski-closed?
\end{question}

\subsection*{Acknowledgements}
We thank Jan Draisma and Rob Eggermont for useful discussions. E.B. and A.O. are partially supported by MIUR and GNSAGA of INdAM (Italy). During the preparation of this manuscript, E.V. was a postdoc at  Universit\"at Bern supported by Vici Grant 639.033.514 of Jan Draisma from the  Netherlands Organisation for Scientific Research. We thank the anonymous referee for their helpful comments which also improved the presentation of the article.  \vspace{-2pt}

\section{Polynomial Functors}\label{sec:polyfunctor}

We now collect the basic notions from the theory of \textit{polynomial functors} of finite degree that we shall need for our example. We restrict our attention to polynomial functors whose elements are tuples of polynomials. We refer to \cite{draisma2019topological,bikgeometry,bik2020phd} for more details on the general theory of polynomial functors. Let $\Vec$ be the category of finite-dimensional vector spaces over a field $\Bbbk$ and let $d\geq1$ be an integer.\vspace{-2pt}

\begin{definition}
The polynomial functor $S^d\colon\Vec\to\Vec$ is the functor that assigns to a finite-dimensional vector space $V\in\Vec$ its $d$-th symmetric power $S^d(V)\in\Vec$ and to a linear map $L\colon V\to W$ the linear map $S^d(L)\colon S^d(V)\to S^d(W)$ it induces.
\end{definition}\vspace{-2pt}

Let {$x_1,\ldots,x_n$} be a basis of $V\in\Vec$. Then $S^d(V)=\Bbbk[x_1,\ldots,x_n]_d$ is the space of degree-$d$ forms. So we say that homogeneous polynomials of degree $d$ are the elements of $S^d$.\vspace{-2pt}

\begin{definition}
Let $P,Q\colon\Vec\to\Vec$ be polynomial functors. Then their direct sum $P\oplus Q\colon\Vec\to\Vec$ is the polynomial functor that assigns to a finite-dimensional vector space $V\in\Vec$ the space $P(V)\oplus Q(V)$ and to a linear map $L\colon V\to W$ the linear map $P(L)\oplus Q(L)\colon P(V)\oplus Q(V)\to P(W)\oplus Q(W)$ sending $(v_1,v_2)\mapsto (P(L)(v_1),Q(L)(v_2))$.
\end{definition}\vspace{-2pt}

So, for all integers $d_1,\ldots,d_k\geq1$, we get a polynomial functor $S^{d_1}\oplus\cdots\oplus S^{d_k}$.\vspace{-2pt}

\begin{definition}
A \textbf{polynomial transformation} $\alpha\colon S^{d_1}\oplus\cdots\oplus S^{d_k}\to S^{e_1}\oplus\cdots\oplus S^{e_\ell}$ consists of a map 
\begin{eqnarray*}
\alpha_V\colon S^{d_1}(V)\oplus\cdots\oplus S^{d_k}(V)&\to& S^{e_1}(V)\oplus\cdots\oplus S^{e_\ell}(V)\\
(f_1,\ldots,f_k)&\mapsto&(F_1(f_1,\ldots,f_k),\ldots,F_\ell(f_1,\ldots,f_k))
\end{eqnarray*}
for every $V\in\Vec$. Here $F_1,\ldots,F_\ell\in \Bbbk[X_1,\ldots,X_k]$ are fixed forms with $\deg(F_j)=e_j$ where $\deg(X_i)=d_i$. We denote by $\im(\alpha)$ the functor $\Vec\to\Set$ that assigns $V\in\Vec$ to the set $\im(\alpha_V)$.
\end{definition}

\begin{remark}\label{remark:polytrans_onetoone}
For fixed $d_1,\ldots,d_k,e_1,\ldots,e_\ell\geq1$, the map that sends such a tuple of polynomials $(F_1,\ldots,F_\ell)$ to the polynomial transformation $S^{d_1}\oplus\cdots\oplus S^{d_k}\to S^{e_1}\oplus\cdots\oplus S^{e_\ell}$ defines a bijection.
\end{remark}

\begin{remark}\label{remark:polytrans_composition}
Given polynomial transformations 
\[
\alpha\colon S^{c_1}\oplus\cdots\oplus S^{c_h}\to S^{d_1}\oplus\cdots\oplus S^{d_k}\mbox{ and }\beta\colon S^{d_1}\oplus\cdots\oplus S^{d_k}\to S^{e_1}\oplus\cdots\oplus S^{e_\ell}
\]
defined by polynomials $F_1,\ldots,F_k$ and $G_1,\ldots,G_\ell$, respectively, note that the composition
\[
\beta\circ\alpha\colon S^{c_1}\oplus\cdots\oplus S^{c_h}\to S^{e_1}\oplus\cdots\oplus S^{e_\ell}
\]
is defined by the polynomials $G_1(F_1,\ldots,F_k),\ldots,G_\ell(F_1,\ldots,F_k)$.
\end{remark}

\begin{definition}
Let $P$ be a polynomial functor. Then its inverse limit is
\[
P_{\infty}:={\textstyle \varprojlim_n} P(\Bbbk^n)=\left\{(f_n)_n\in\prod_{n=1}^\infty P(\Bbbk^n)\,\middle|\, P(\pi_n)(f_{n+1})=f_n\mbox{ for all }n\geq 1\right\}
\]
where $\pi_n\colon \Bbbk^{n+1}\to \Bbbk^n$ is the projection map that forgets the last coordinate.
\end{definition}

Let $d\geq1$ be an integer. Then the elements of $S^d_{\infty}$ are polynomial series
\[
f={\textstyle \sum_{\underline{e}}c_{\underline{e}}}x^{\underline{e}},\quad c_{\underline{e}}\in \Bbbk,
\]
where $\underline{e}=(e_1,e_2,\ldots)$ ranges over all sequences of nonnegative integers that sum up to $d$ and $x^{\underline{e}}=x_1^{e_1}x_2^{e_2}\cdots$. The set $\Bbbk\oplus \bigoplus_{e=1}^\infty S^e_{\infty}$ naturally has the structure of a graded $\Bbbk$-algebra. The following theorem of Erman, Sam and Snowden tells us that this $\Bbbk$-algebra is in fact a polynomial ring.

\begin{theorem} \cite[Theorem 1.1]{erman2020big}
There exists an index set $I$ and a map $d\colon I\to\NN$ such that $\Bbbk\oplus \bigoplus_{e=1}^\infty S^e_{\infty}$ is isomorphic as a graded $\Bbbk$-algebra to the polynomial ring $\Bbbk[y_i\mid i\in I]$ where $\deg(y_i)=d(i)$.
\end{theorem}

\begin{remark}
The notion of strength naturally extends both to series in $S^d_{\infty}$ and to homogeneous polynomials in $\Bbbk[y_i\mid i\in I]$. In both cases, we define the strength of an element $f$ to be the infimal length of a strength decomposition of $f$. When $f$ has a strength decomposition, this definition coincides with Definition~\ref{def:strength}. When $f$ has no strength decomposition, we instead say that $f$ has infinite strength.
Let 
\[
\phi\colon\Bbbk[y_i\mid i\in I]\to\Bbbk\oplus \bigoplus_{e=1}^\infty S^e_{\infty}
\]
be any graded $\Bbbk$-algebra isomorphism and let $f\in \Bbbk[y_i\mid i\in I]$ be a homogeneous polynomial of degree $d$. Then the strengths of $f$ in $\Bbbk[y_i\mid i\in I]$ and $\phi(f)$ in $S^d_{\infty}$ coincide.  
\end{remark}

\begin{proposition}
Let $f\in \Bbbk[y_i\mid i\in I]$ be a homogeneous polynomial of degree $d$. Then $f$ has finite strength if and only if $f\in \Bbbk[y_i\mid i\in I, \deg(y_i)<d]$.
\end{proposition}
\begin{proof}
Suppose that $f$ has finite strength. Then all terms in a strength decomposition of $f$ have degree $<d$ and are therefore contained in $\Bbbk[y_i\mid i\in I, \deg(y_i)<d]$. Hence $f$ is as well. 

Suppose that $f\in \Bbbk[y_i\mid i\in I, \deg(y_i)<d]$. Then the polynomial $f$ is a sum of monomials in variables $y_i$ of degree $<d$. Since $f$ has degree $d$, each of these monomials must be reducible. This yields a strength decomposition of $f$. Hence $f$ has finite strength.
\end{proof}

\begin{remark}
The proposition shows that all variables $y_i$ have infinite strength. In particular, the variables with degree $\geq2$. So in this setting, not all polynomials of degree $\geq 2$ have finite strength.
\end{remark}

\begin{definition}
A tuple of series
\[
(f_{e,j})_{e,j}\in \bigoplus_{e=1}^d  (S^e_{\infty})^{\oplus k_e}
\]
is \textbf{part of a system of variables (over $\Bbbk$)} when $(\phi^{-1}(f_{e,j}))_{e,j}$ is a tuple of distinct variables for some graded $\Bbbk$-algebra isomorphism $\phi\colon\Bbbk[y_i\mid i\in I]\to\Bbbk\oplus \bigoplus_{e=1}^\infty S^e_{\infty}$.
\end{definition}

\begin{proposition}\label{prop:characterization_systemVars}
A tuple of series
\[
(f_{e,j})_{e,j}\in \bigoplus_{e=1}^d  (S^e_{\infty})^{\oplus k_e}
\]
is part of a system of variables if and only if every element of
\[
\{\lambda_1f_{e,1}+\ldots+\lambda_{k_e}f_{e,k_e}\mid (\lambda_1:\cdots:\lambda_{k_e})\in\PP^{k_e-1}\}
\]
has infinite strength for all $e\in\{1,\ldots,d\}$. 
\end{proposition}
\begin{proof}
Let $\phi\colon\Bbbk[y_i\mid i\in I]\to\Bbbk\oplus \bigoplus_{e=1}^\infty S^e_{\infty}$ be a graded $\Bbbk$-algebra isomorphism. Suppose that $(\phi^{-1}(f_{e,j}))_{e,j}$ is a tuple of distinct variables. Then
\[
\lambda_1\phi^{-1}(f_{e,1})+\ldots+\lambda_{k_e}\phi^{-1}(f_{e,k_e})
\]
is not a polynomial in variables of degree $<e$ for every $(\lambda_1:\cdots:\lambda_{k_e})\in\PP^{k_e-1}$. Hence 
\[
\lambda_1f_{e,1}+\ldots+\lambda_{k_e}f_{e,k_e}
\]
has infinite strength for every $(\lambda_1:\cdots:\lambda_{k_e})\in\PP^{k_e-1}$. 

We prove the inverse statement using induction on $d$. So we may assume that
\[
(f_{e,j})_{e<d,j}\in \bigoplus_{e=1}^{d-1}  (S^e_{\infty})^{\oplus k_e}
\]
is part of a system of variables and $(\phi^{-1}(f_{e,j}))_j$ is a tuple of distinct variables of degree $e$ for all $e<d$. Hence, it is enough to construct a change of variables in $\Bbbk[y_i \mid i \in I]$ which is the identity on the $y_i$'s with degree $<d$ and turns $(\phi^{-1}(f_{d,j}))_j$ into a tuple of variables. Write
\[
\phi^{-1}(f_{d,j})=\ell_j+g_j
\]
where $\ell_j$ is a finite linear combination of the variables $y_i$ of degree $d$ and $g_j\in\Bbbk[y_i\mid i\in I,\deg(y_i)<d]$. Suppose that $\lambda_1\ell_1+\ldots+\lambda_{k_d}\ell_{k_d}=0$ for some $(\lambda_1:\cdots:\lambda_{k_d})\in\PP^{k_d-1}$. Then
\[
\lambda_1f_{d,1}+\ldots+\lambda_{k_d}f_{d,k_d}=\varphi(\lambda_1g_1+\ldots+\lambda_{k_d}g_{k_d})
\]
has finite strength, because $\lambda_1g_1+\ldots+\lambda_{k_d}g_{k_d}$ lies in $\Bbbk[y_i\mid i\in I,\deg(y_i)<d]$ and hence has finite strength. Since this is not the case, we see that $\ell_1,\ldots,\ell_{k_d}$ must be linearly independent. So there exists a graded $\Bbbk$-algebra automorphism $\psi$ of $\Bbbk[y_i\mid i\in I]$ sending $y_i\mapsto y_i$ when $\deg(y_i)\neq d$ such that $\psi^{-1}(\ell_j)$ is a variable for $j=1,\ldots,k_d$. Replacing $\phi$ by $\phi\circ\psi$, we may assume that the $\ell_j$ are already variables. Now, let $\omega$ be the automorphism $\Bbbk[y_i\mid i\in I]$ sending $\ell_j\mapsto\ell_j+g_j$ and sending all other variables to themselves. Then the isomorphism $\phi\circ\omega$ shows that $(f_{e,j})_{e,j}$ is part of a system of variables.
\end{proof}

The next lemma shows that there exist tuples defined over $\Bbbk$ that are part of a system of variables over any field extension of $\Bbbk$.

\begin{lemma}\label{lm:part_exists}
For $e=1,\ldots,d$, let $\pi_e\colon \{1,\ldots,k_e\}\times\{1,\ldots,e\}\times\NN\to\NN$ be an injection and take
\[
f_{e,j}=\sum_{i=1}^\infty x_{\pi_e(j,1,i)}\cdots x_{\pi_e(j,e,i)}\in S^e_{\infty}.
\]
Then $(f_{e,j})_{e,j}$ is part of a system of variables.
\end{lemma}
\begin{proof}
Let $\phi\colon\Bbbk[y_i\mid i\in I]\to\Bbbk\oplus \bigoplus_{e=1}^\infty S^e_{\infty}$ be a graded $\Bbbk$-algebra isomorphism and suppose that 
\[
\lambda_1f_{e,1}+\ldots+\lambda_{k_e}f_{e,k_e}
\]
has finite strength for some $(\lambda_1:\cdots:\lambda_{k_e})\in\PP^{k_e-1}$. Then 
\[
\phi^{-1}(\lambda_1f_{e,1}+\ldots+\lambda_{k_e}f_{e,k_e})\in \Bbbk[y_i\mid i\in I,\deg(y_i)<e]
\]
Let $f\in \Bbbk[y_i\mid i\in I,\deg(y_i)<e]$ be a homogeneous polynomial of degree $e$. By relabelling the variables, we may assume that $f=f(y_1,\ldots,y_k)$ for some variables $y_1,\ldots,y_k$ of degree $<e$. The product rule shows that
\[
\phi^{-1}\left(\frac{\partial}{\partial x_j}\phi(f)\right)\in (y_1,\ldots,y_k)
\]
for each $j\in\NN$. Hence, in order to get a contradiction, it suffices to prove for all $(\lambda_1:\cdots:\lambda_{k_e})\in\PP^{k_e-1}$ that
\[
\left\{\phi^{-1}\left(\frac{\partial}{\partial x_j}(\lambda_1f_{e,1}+\ldots+\lambda_{k_e}f_{e,k_e})\right)\,\middle|\, j\in\NN\right\}
\]
is not contained in an ideal of $\Bbbk[y_i\mid i\in I]$ generated by finitely many variables. This is indeed the case since, by construction, the above set consists up to scaling of infinitely many monomials in pairwise distinct variables. So, by Proposition \ref{prop:characterization_systemVars}, we see that $(f_{e,j})_{e,j}$ is part of a system of variables.
\end{proof}

Before we explain our proof strategy, we first give the intuition behind it. Let $h\in\Bbbk[y_1,\ldots,y_n]$ be a homogeneous polynomial of degree $d\geq2$ where we have $\deg(y_i)=d_i>0$. Let $f_1,\dots,f_n$ be forms of degrees $d_1,\ldots,d_n$ in variables $x_1,\ldots,x_m$ of degree $1$. Then we can consider the form $h(f_1,\ldots,f_n)$.  Suppose that $h$ has a strength decomposition, say of length $r$. Then, by evaluating this strength decomposition in $f_1,\ldots,f_n$, we get a strength decomposition of $h(f_1,\ldots,f_n)$. We see that $\str(h(f_1,\ldots,f_n))\leq\str(h)$ for all $f_1,\ldots,f_n$ whenever the latter is finite. In general, we have no reason to expect equality to hold; indeed, we also have $\str(h(f_1,\ldots,f_n))\leq\str(f_1)+\ldots+\str(f_n)$ and $\str(h(f_1,\ldots,f_n))\leq m$ which may yield far stronger bounds in certain cases. 

Now, the idea behind our proof is that the strength of $h$ is much easier to compute than the strength of $h(f_1,\ldots,f_n)$. So we first go to a setting where $f_1,\ldots,f_n$ can be treated as if they are variables and so the strength of these two forms is in fact the same. This setting is precisely the case where $(f_1,\ldots,f_n)$ is part of a system of variables. We then translate this back to statements about polynomials.

Next, we explain our proof strategy in detail. It is inspired by the theory of polynomial functors from \cite{bikgeometry} established by Bik, Draisma, Eggermont, and Snowden. A polynomial transformation $\alpha\colon Q\to P$ naturally induces a map $\alpha_{\infty}\colon Q_{\infty}\to P_{\infty}$. \vspace{-2pt}

\begin{proposition}\label{mainpolyfunctor}
Take $P=S^{d_1}\oplus\cdots\oplus S^{d_k}$ and $Q=S^{e_1}\oplus\cdots\oplus S^{e_\ell}$. Let $\alpha\colon P\to S^d$ and $\beta\colon Q\to S^d$ be polynomial transformations and let $p=(p_1,\ldots,p_k)\in P_{\infty}, q=(q_1,\ldots,q_\ell)\in Q_{\infty}$. Suppose that $q$ is part of a system of variables and $\alpha_{\infty}(p)=\beta_{\infty}(q)$. Then $\beta=\alpha\circ\gamma$ for some polynomial transformation $\gamma\colon Q\to P$.
\end{proposition}
\begin{proof}
When $\mathrm{char}(\Bbbk)=0$, this is a special case of \cite[Proposition 4.5.17]{bik2020phd}. Let $\phi\colon\Bbbk[y_i\mid i\in I]\to\Bbbk\oplus \bigoplus_{e=1}^\infty S^e_{\infty}$ be a graded $\Bbbk$-algebra isomorphism such that $(\phi^{-1}(q_1),\ldots,\phi^{-1}(q_\ell))$ is a tuple of distinct variables. By relabelling the variables, we may assume that $\phi^{-1}(q_i)=y_i$. We have $\beta_{\infty}(q)=\phi(B(y_1,\ldots,y_\ell))$ and $\alpha_{\infty}(p)=\phi(A(z_1,\ldots,z_k))$ for some polynomials $A,B$ where $z_j=\phi^{-1}(p_j)\in\Bbbk[y_i\mid i\in I]$. As the $z_j$ are elements of $\Bbbk[y_i\mid i\in I]$, by relabelling more variables, we may assume that they are polynomials in $y_1,\ldots,y_{\ell'}$ for some $\ell'\geq \ell$. y the assumption $\alpha_\infty(p) = \beta_\infty(q)$, we have
\[
A(z_1(y_1,\ldots,y_{\ell'}),\ldots,z_k(y_1,\ldots,y_{\ell'}))=B(y_1,\ldots,y_\ell).
\]
Now, we set $y_i=0$ for $i>\ell$ on both sides and find that
\[
A(z_1(y_1,\ldots,y_{\ell},0,\ldots,0),\ldots,z_k(y_1,\ldots,y_{\ell},0,\ldots,0))=B(y_1,\ldots,y_\ell).
\]
Write $F_j=z_j(y_1,\ldots,y_{\ell},0,\ldots,0)\in\Bbbk[y_1,\ldots,y_\ell]$. Then we see that $\beta=\alpha\circ\gamma$ for the polynomial transformation $\gamma\colon Q\to P$ defined by $(F_1,\ldots,F_k)$.
\end{proof}\vspace{-2pt}

In other words, if $\beta_{\infty}(q)\in\im(\alpha_\infty)$, then $\beta$ factors through $\alpha$. To relate this condition back to polynomials, we have the following lemma.\vspace{-2pt}

\begin{lemma}\label{image}
Suppose that $\Bbbk$ is algebraically closed and let $\mathbb{L}$ be an uncountable algebraically closed extension of $\Bbbk$. Take $P=S^{d_1}\oplus\cdots\oplus S^{d_k}$, let $\alpha\colon P\to S^d$ be a polynomial transformation and let $f\in S^d_{\infty}$ be the inverse limit of a sequence $(f_n)_n\in\prod_{n=1}^{\infty}S^d(\Bbbk^n)$. Assume that $f=\beta_{\infty}(q)$ for some polynomial transformation $\beta\colon Q\to S^d$ where $Q=S^{e_1}\oplus\cdots\oplus S^{e_\ell}$ and $q\in Q_{\infty}$ is part of system of variables over $\mathbb{L}$. Then $f\in\im(\alpha_\infty)$ if and only if $f_n\in\im(\alpha_{\Bbbk^n})$ for all~$n\in\NN$.
\end{lemma}
\begin{proof}
When $\mathrm{char}(\Bbbk)=0$, this is a special case of \cite[Lemma 4.5.24]{bik2020phd}. We follow its proof. Clearly, if $f\in\im(\alpha_\infty)$, then $f_n\in\im(\alpha_{\Bbbk^n})$ for all~$n\in\NN$. So assume that $f_n\in\im(\alpha_{\Bbbk^n})$ for all~$n\in\NN$.

Before we prove the general case, we first consider the case where $\mathbb{L} = \Bbbk$. Let $p=(p_n)_n\in P_{\infty}$. Then the equality $\alpha_\infty(p)=f$ holds if and only if $\alpha_{\Bbbk^n}(p_n)=f_n$ holds for all $n\in\NN$. This translates the condition $\alpha_\infty(p)=f$ into polynomial equations in countably many variables and the condition that $f_n\in\im(\alpha_{\Bbbk^n})$ for all~$n\in\NN$ shows that any finite number of these equations has a solution. Hence, by Lang's theorem from \cite{lang} the entire system has a solution since $\Bbbk$ is algebraically closed and uncountable.

Now for the general case, note that by the first part of the proof, there exists a $p\in P_{\infty}^{\mathbb{L}}$ defined over $\mathbb{L}$ such that $\alpha_{\infty}^{\mathbb{L}}(p)=f$. We now have $\alpha_{\infty}^{\mathbb{L}}(p)=\beta_{\infty}^{\mathbb{L}}(q)$. By Proposition \ref{mainpolyfunctor}, it follows that $\beta^{\mathbb{L}}=\alpha^{\mathbb{L}}\circ\gamma^{\mathbb{L}}$ for some polynomial transformation $\gamma^{\mathbb{L}}\colon Q^{\mathbb{L}}\to P^{\mathbb{L}}$ defined over $\mathbb{L}$. The condition $\beta=\alpha\circ\gamma$ defines a Zariski-closed subset in the finite-dimensional space of polynomial transformations $\gamma\colon Q\to P$: since we have just observed that this Zariski-closed set is non-empty over $\mathbb{L}$ and $\Bbbk$ is algebraically closed, then it must be non-empty also over $\Bbbk$. In particular, we get that $\beta=\alpha\circ\gamma$ for some $\gamma\colon Q\to P$. The point $p=\gamma_\infty(q)$ satisfies $\alpha_\infty(p)=f$.
\end{proof}

This gives us the following proof strategy: let 
\[
\alpha\colon S^{d_1}\oplus\cdots\oplus S^{d_k}\to S^{d}\mbox{ and }\beta\colon S^{e_1}\oplus\cdots\oplus S^{e_\ell}\to S^{d}
\]
be polynomial transformations defined by polynomials $F$ and $G$ respectively. If $\beta$ factors through $\alpha$, then $G=F(H_1,\ldots,H_k)$ for some polynomials $H_1,\ldots,H_k$. So if we can prove this not to be the case, then $f:=\beta_{\infty}(q)\not\in\im(\alpha_\infty)$ for any point $q$ that is part of a system of variables over all field extensions of $\Bbbk$ and hence we get $f_n\not\in\im(\alpha_{\Bbbk^n})$ for the projection $f_n$ of such an $f$ to $S^d(\Bbbk^n)$ for some integer $n\geq1$. Our goal now is to choose $\alpha,\beta$ such that this conclusion is exactly what we want.

\begin{remark}\label{increasingn}
Since the diagram
\[
\xymatrix{
S^{d_1}(V)\oplus\cdots\oplus S^{d_k}(V)\ar[rr]^{\alpha_V}\ar[d]^{S^{d_1}(L)\oplus\cdots\oplus S^{d_k}(L)}&&S^{d}(V)\ar[d]^{S^{d}(L)}&&S^{e_1}(V)\oplus\cdots\oplus S^{e_\ell}(V)\ar[ll]_{\beta_V}\ar[d]^{S^{e_1}(L)\oplus\cdots\oplus S^{e_\ell}(L)}\\
S^{d_1}(W)\oplus\cdots\oplus S^{d_k}(W)\ar[rr]_{\alpha_W}&&S^{d}(W)&&S^{e_1}(W)\oplus\cdots\oplus S^{e_\ell}(W)\ar[ll]^{\beta_W}
}
\]
commutes for all linear maps $L\colon V\to W$, it in this case follows that $f_n\not\in\im(\alpha_{\Bbbk^n})$ for all $n\gg1$ by choosing for $L$ the projection maps $\pi_n$. More precisely, for $L=\pi_n$, the diagram shows that if $f_{n+1}\in\im(\alpha_{\Bbbk^{n+1}})$, then also $f_n\in\im(\alpha_{\Bbbk^n})$.
\end{remark}

\section{The example in a finite setting}

Before explaining our example, we prove a theorem which will be the heart of the proof in the next section. Here, we consider polynomials in the polynomial ring $\Bbbk[x,y,u,v,f,g,p,q]$ where $x,y,u,v$ and $f,g,p,q$ are variables of degrees $1$ and $2$, respectively. We start by defining the strength of a homogeneous polynomial in this setting.

\begin{definition}
A \textbf{strength decomposition} of a homogeneous polynomial $h\in\Bbbk[x,y,u,v,f,g,p,q]$ is an expression of the form
\begin{equation}
h = g_1h_1 + \ldots + g_rh_r, \quad \text{ where $g_i,h_i$ are homogeneous with } 1\leq \deg(g_i),\deg(h_i) \leq \deg(h)-1.
\end{equation}
We define the \textbf{strength} of $h$ to be the smallest length of a strength decomposition of $h$.
\end{definition}

\begin{theorem}\label{thm:strength=4}
The polynomial $x^2f+y^2g+u^2p+v^2q\in \Bbbk[x,y,u,v,f,g,p,q]$ has strength $4$.
\end{theorem}
\begin{proof}
The polynomial $x^2f+y^2g+u^2p+v^2q$ has strength $\leq 4$. We need to show that $x^2f+y^2g+u^2p+v^2q$ has no strength decomposition of length $3$. This gives us four cases.  We will use the following notation. For $i\in\{1,2,3\}$, let $x_i,g_i,h_i,q_i\in \Bbbk[x,y,u,v,f,g,p,q]$ be homogeneous polynomials of degrees $1,2,2,3$, respectively. Let $R = \Bbbk[x,y,u,v]$. Then, $x_i \in R$ and
\begin{align*}
g_i&=a_{i,1}f+a_{i,2}g+a_{i,3}p+a_{i,4}q+\hat{g}_i,\\
h_i&=b_{i,1}f+b_{i,2}g+b_{i,3}p+b_{i,4}q+\hat{h}_i,\\
q_i&=c_{i,1}f+c_{i,2}g+c_{i,3}p+c_{i,4}q
\end{align*}
for some $a_{i,j},b_{i,j}\in\Bbbk$ and $c_{i,j},\hat{g}_i,\hat{h}_i\in R$. Write
\[
G_i := a_{i,1}f+a_{i,2}g+a_{i,3}p+a_{i,4}q\quad\mbox{ and }\quad H_i:=b_{i,1}f+b_{i,2}g+b_{i,3}p+b_{i,4}q.
\]

\textit{Case a.}
We show that 
\begin{equation}\label{eq:case_a}
x^2f+y^2g+u^2p+v^2q\neq x_1q_1+x_2q_2+x_3q_3.
\end{equation}
View both sides as polynomials in $R[f,g,p,q]$. So on the left-hand side, the coefficients of $f,g,p,q$ are $x^2,y^2,u^2,v^2$. Recall that $x_i\in R$ and 
\[
q_i=c_{i,1}f+c_{i,2}g+c_{i,3}p+c_{i,4}q
\]
for some homogeneous $c_{i,j}\in R$ of degree $1$. Note that, on the right-hand side of \eqref{eq:case_a}, the coefficient $x_1c_{1,1}+x_2c_{2,1}+x_3c_{3,1}$ of $f$ is contained in the ideal $(x_1,x_2,x_3)\subseteq  R$. Assume by contradiction that in \eqref{eq:case_a} equality holds. This implies that $x^2\in (x_1,x_2,x_3)$ and then $x\in(x_1,x_2,x_3)$ since the ideal is prime. Similarly, $y,u,v\in(x_1,x_2,x_3)$. But this is impossible, as the four variables $x,y,u,v$ cannot all lie in an ideal generated by three linear forms. Hence \eqref{eq:case_a} is indeed an inequality.

\textit{Case b.}
We show that 
\begin{equation}\label{eq:case_b}
x^2f+y^2g+u^2p+v^2q\neq x_1q_1+x_2q_2+g_1h_1.
\end{equation}
In the notation introduced above, if $G_1H_1\neq0$, then we see that the coefficient of one of the monomials $f^2,fg,\ldots,pq,q^2$ in the right-hand-side of \eqref{eq:case_b} is nonzero. Therefore, \eqref{eq:case_b} holds in this case. If $G_1=0$, we see that the coefficients of $f,g,p,q$ on the right-hand side of \eqref{eq:case_b} are contained in the ideal $(x_1,x_2,\hat{g}_1)\subseteq R$. Similarly, if $H_1=0$, the coefficients of $f,g,p,q$ on the right-hand side of \eqref{eq:case_b} are contained in the ideal $(x_1,x_2,\hat{h}_1)\subseteq R$. In both cases, this is impossible as these ideals cannot contain all the powers $x^2,y^2,u^2,v^2$ by Krull's height theorem. Hence, $G_1H_1 \neq 0$ and then \eqref{eq:case_b} holds.

\textit{Case c.}
We show that 
\begin{equation}\label{eq:case_c}
x^2f+y^2g+u^2p+v^2q\neq x_1q_1+g_1h_1+g_2h_2.
\end{equation}
Let $a_{i,j},b_{i,j}\in\Bbbk$, $c_{i,j},\hat{g}_i,\hat{h}_i\in R$ and $G_i,H_i$ be as before. Assume by contradiction that \eqref{eq:case_c} is instead an equality and set $x,y,u,v=0$. Then we get $0=G_1H_1+G_2H_2$. It follows that after reordering and scaling, we have $(G_2,H_2)=(G_1,-H_1)$ or $G_1=G_2=0$. In the first case, we find that the coefficients of $f,g,h,q$ on the right-hand side of \eqref{eq:case_c} are contained in $(x_1,\hat{g}_1-\hat{g}_2,\hat{h}_1+\hat{h}_2)$. In the second case, we find that these coefficients are contained in $(x_1,\hat{g}_1,\hat{g}_2)$. Both these cases are impossible since $x^2,y^2,u^2,v^2$ cannot all be contained in such ideals by Krull's height theorem. Hence \eqref{eq:case_c} holds.

\textit{Case d.}
We show that 
\begin{equation}\label{eq:case_d}
x^2f+y^2g+u^2p+v^2q\neq g_1h_1+g_2h_2+g_3h_3.
\end{equation}
Let $a_{i,j},b_{i,j}\in\Bbbk$, $\hat{g}_i,\hat{h}_i\in R$ and $G_i,H_i$ be as before. Assume by contradiction that \eqref{eq:case_d} is instead an equality. First, consider both sides of \eqref{eq:case_d} as polynomials in $x,y,u,v$ with coefficients in $\Bbbk[f,g,p,q]$. Then the coefficients of $x^2,y^2,u^2,v^2$ on the right-hand side are contained in the span of $G_1,H_1,G_2,H_2,G_3,H_3$. As these coefficients are $f,g,p,q$ on the left-hand side, we see that this span must be $4$-dimensional. So among $G_1,H_1,G_2,H_2,G_3,H_3$ there must be four linearly independent forms. After reordering, we may assume that these forms are either $G_1,H_1,G_2,H_2$ or $G_1,H_1,G_2,G_3$. In both cases, we call these forms $F,G,P,Q$ and note that the remaining two forms are also forms in $F,G,P,Q$. Now, we set $x,y,u,v=0$. We get the equation $0=G_1H_1+G_2H_2+G_3H_3$. So we see that either $FG+PQ+G_3H_3=0$ or $FG+PH_2+QH_3=0$. Both of these equations have no solutions: indeed, for the first, $FG+PQ$ is irreducible; for the second, we see that the equality $FG = -PH_2-QH_3$ cannot hold by setting $P=Q=0$. 
\end{proof}

\section{The example}\label{sec:example}

We are now ready to explain our example which gives a negative answer to Question \ref{quest}. Assume that $\Bbbk$ is algebraically closed. For any $t \in \Bbbk\setminus\{0\}$, consider 
\begin{equation}\label{strengthborder}
h(t):=\frac{1}{t}\left((x^2+tg)(y^2+tf)-(u^2-tq)(v^2-tp)-(xy+uv)(xy-uv)\right)
\end{equation}
where $x,y,u,v$ and $f,g,p,q$ are forms in $n$ variables of degrees $1$ and $2$, respectively. One sees that
\begin{equation}\label{eq:polynomial_p}
h :=\lim_{t\to 0}h(t)=x^2f+y^2g+u^2p+v^2q
\end{equation}
is a limit of strength-$\leq3$ polynomials and has itself strength $\leq 4$. We use the machinery of polynomial functors from Section~\ref{sec:polyfunctor} to show that, if $n \gg 0$, the polynomial $h$ indeed has strength $4$ for some choices of $x,y,u,v,f,g,p,q$.

Consider the following polynomial transformations:
\[\begin{array}{rlccc}
\alpha_1\colon&(S^1)^{\oplus 3}\oplus (S^3)^{\oplus 3}&\to& S^4,&(x_1,x_2,x_3,q_1,q_2,q_3)\mapsto x_1q_1+x_2q_2+x_3q_3,\\
\alpha_2\colon&(S^1)^{\oplus 2}\oplus (S^3)^{\oplus 2}\oplus S^2\oplus S^2&\to& S^4,&(x_1,x_2,q_1,q_2,g_1,h_1)\mapsto x_1q_1+x_2q_2+g_1h_1,\\
\alpha_3\colon&(S^1)\oplus (S^3)\oplus(S^2)^{\oplus 2}\oplus(S^2)^{\oplus 2}&\to& S^4,&(x_1,q_1,g_1,g_2,h_1,h_2)\mapsto x_1q_1+g_1h_2+g_2h_2,\\
\alpha_4\colon&(S^2)^{\oplus 3}\oplus(S^2)^{\oplus 3}&\to& S^4,&(g_1,g_2,g_3,h_1,h_2,h_3)\mapsto g_1h_1+g_2h_2+g_3h_3,\\
\beta\colon&(S^1)^{\oplus 4}\oplus (S^2)^{\oplus 4}&\to& S^4,&(x,y,u,v,f,g,p,q)\mapsto x^2f+y^2g+u^2p+v^2q. 
\end{array}\]
The image $\im(\beta)$ form the set of polynomials as in \eqref{eq:polynomial_p} over all $x,y,u,v,f,g,p,q$. Notice that $\bigcup_{i=1}^4\im(\alpha_i)$ consists of all forms of degree $4$ with strength $\leq 3$. By the discussion in Section~\ref{sec:polyfunctor}, in order to prove that $h$ has strength $>3$ for some choice of $x,y,u,v,f,g,p,q$, it suffices to prove that $\beta$ does not factor via $\alpha_i$ for any $i\in\{1,2,3,4\}$. This is essentially the statement of Theorem~\ref{thm:strength=4}.

\begin{lemma}\label{lm:case}
(a)
There is no $\gamma\colon (S^1)^{\oplus 4}\oplus (S^2)^{\oplus 4}\to (S^1)^{\oplus 3}\oplus (S^3)^{\oplus 3}$ such that $\beta=\alpha_1\circ\gamma$.

(b)
There is no $\gamma\colon (S^1)^{\oplus 4}\oplus (S^2)^{\oplus 4}\to (S^1)^{\oplus 2}\oplus (S^3)^{\oplus 2}\oplus S^2\oplus S^2$ such that $\beta=\alpha_2\circ\gamma$.

(c)
There is no $\gamma\colon (S^1)^{\oplus 4}\oplus (S^2)^{\oplus 4}\to (S^1)\oplus (S^3)\oplus(S^2)^{\oplus 2}\oplus(S^2)^{\oplus 2}$ such that $\beta=\alpha_3\circ\gamma$.

(d)
There is no $\gamma\colon (S^1)^{\oplus 4}\oplus (S^2)^{\oplus 4}\to (S^2)^{\oplus 3}\oplus(S^2)^{\oplus 3}$ such that $\beta=\alpha_4\circ\gamma$.
\end{lemma}
\begin{proof}
By Remarks~\ref{remark:polytrans_onetoone} and~\ref{remark:polytrans_composition}, if any of (a)-(d) would not hold, then the polynomial $x^2f+y^2g+u^2p+v^2q\in \Bbbk[x,y,u,v,f,g,p,q]$ would have strength $\leq3$. Hence Theorem~\ref{thm:strength=4} implies the lemma.
\end{proof}

By Lemma \ref{lm:case}, Proposition \ref{mainpolyfunctor}, Lemma \ref{lm:part_exists}, Lemma \ref{image} and Remark~\ref{increasingn}, we conclude that the form~$h$ in~\eqref{eq:polynomial_p} has strength $4$ for some choice of $x,y,u,v,f,g,p,q$. So the set of forms of degree $4$ in $n\gg0$ variables with strength $\leq 3$ is not closed. 

\begin{remark}
In the proof of Theorem \ref{thm:strength=4}, the symbols $x,y,u,v,f,g,h,q$ are treated as {\it variables}, the first four of degree $1$ and the last four of degree $2$. We think of $(x,y,u,v,f,g,h,q)$ as a tuple of polynomial series that is part of a system of variables. So this proof strategy {\it does not} directly apply for a fixed value of $n$ (i.e., in a fixed ambient dimension), and so does not apply towards answering Question \ref{quest1}. More specifically, once $n$ is fixed, the symbols $f, g, h, q$ would not be formal variables of degree $2$ anymore, but rather degree $2$ forms in $n$ variables. 
\end{remark}

\begin{remark}
Similar to all problems concerning additive decompositions of forms, one of the biggest challenges is to provide efficient methods to get lower bounds on the strength of a given form $f$. As far as we know, the only general method is observed in \cite[Remark 1.1]{ananyan2020strength}: if $\mathrm{Sing}(f)$ is the singular locus of the hypersurface $\{f=0\}$ and  $\codim(\Sing(f))\geq 2k+1$, then $\str(f)\geq k+1$. Suppose we have a flat family of hypersurfaces $f(t) \rightarrow f(0)=f$ such that $\str(f(t))\leq k$ for all $t\in \mathbb A^1_{\Bbbk}\setminus\{0\}$.  By the aforementioned bound, one has $\codim(\Sing(f(t)))\leq 2k$. Note that, in such a situation, $\codim(\Sing(f(t)))\geq \codim(\Sing(f))$. Therefore, the Ananyan-Hochster lower bound cannot be better than~$k$. In our example, $\codim(\Sing(h))\leq 4$ and so this lower bound is not larger than $2$.
\end{remark}

\end{document}